\documentclass[12pt]{amsart}
\usepackage{amsmath}
\usepackage{amsfonts}
\usepackage{amssymb}
\usepackage{amsthm}
\usepackage{tikz}
\usetikzlibrary{calc}
\usepackage{enumitem}
\usepackage{url}
\usepackage{float}

\numberwithin{equation}{section}
\newtheorem{theorem}{Theorem}[section]
\newtheorem{prop}[theorem]{Proposition}
\newtheorem{lemma}[theorem]{Lemma}
\newtheorem{cor}[theorem]{Corollary}

\newtheorem{alg}[theorem]{Algorithm}
\newtheorem{definition}[theorem]{Definition}

\theoremstyle{remark}

\newtheorem{example}[theorem]{Example}
\newtheorem{remark}[theorem]{Remark}

\newcommand{\C}{\mathcal{C}}
\newcommand{\defeq}{\mathrel{\mathop:}=}

\newcommand{\R}{1.8} 
\newcommand{\n}{8} 

\newcommand{\circpic}[9]{
\begin{tikzpicture}
        \node(center) at (0,0){#1};
        
        \foreach \nr in {1, ..., \n}{
            \path (360/\n*\nr:\R) coordinate (i\nr);
        }
        \foreach \nr in {8,3,5}{
            \draw (center) -- (i\nr) node[midway, above] {};
        }
        
        \draw[white] (center) -- (i4) node[black,midway] {$(1,2)^\prime$};
        \draw[white] (center) -- (i7) node[black, midway, left] {$(1,2)$};
        \draw[red] (i1) \foreach \i in {2,...,\n} {-- (i\i)} -- cycle;
        
        \draw (center) circle (\R);   
        
        \node at (1.9956, .8268) {#2};
        \node at (.8268, 1.9956) {#3};
        \node at (-.8268, 1.9956) {#4};
        \node at (-1.9956, .8268) {#5};
        \node at (-1.9956, -.8268) {#6};
        \node at (-.8268, -1.9956) {#7};
        \node at (.8268, -1.9956) {#8};
        \node at (1.9956, -.8268) {#9};
        \end{tikzpicture}
}

\newcommand{\circpicnodiv}[9]{
	\begin{tikzpicture}
	\node(center) at (0,0){#1};
	
	\foreach \nr in {1, ..., \n}{
		\path (360/\n*\nr:\R) coordinate (i\nr);
	}

	\draw[red] (i1) \foreach \i in {2,...,\n} {-- (i\i)} -- cycle;
	
	\draw (center) circle (\R);   
	
	\node at (1.9956, .8268) {#2};
	\node at (.8268, 1.9956) {#3};
	\node at (-.8268, 1.9956) {#4};
	\node at (-1.9956, .8268) {#5};
	\node at (-1.9956, -.8268) {#6};
	\node at (-.8268, -1.9956) {#7};
	\node at (.8268, -1.9956) {#8};
	\node at (1.9956, -.8268) {#9};
	\end{tikzpicture}
}
   
\newcommand{\circpicex}[4]{
	\begin{tikzpicture}
	\node(center) at (0,0){#1};
	
	\foreach \nr in {1, 2, 3}{
		\path (360/3*\nr:\R) coordinate (i\nr);
	}

	\draw[red] (i1) \foreach \i in {2,3} {-- (i\i)} -- cycle;
	
	\draw (center) circle (\R);   
	
	\node at (.75, 1.2) {#2};
	\node at (-1.4, 0) {#3};
	\node at (.75, -1.2) {#4};
	\end{tikzpicture}
}   

\title{Strongly maximal intersection-complete neural codes on grids are convex}
\author{Robert Williams}
\address{Robert Williams, Department of Mathematics, Texas A\&M University, College Station, Texas 77843, USA}
\email{rwilliams@math.tamu.edu}
\urladdr{\url{http://www.math.tamu.edu/~rwilliams}}

\begin{document}

\begin{abstract}
The brain encodes spacial structure through a combinatorial code of neural
activity. Experiments suggest such codes  correspond to convex areas of the subject's environment. We present an intrinsic condition that implies a
neural code may correspond to a convex space and give a bound on the minimal dimension underlying such a realization.
\end{abstract}

\maketitle

\section{Introduction}

The brain is continuously interpreting external stimuli to navigate the physical world around it. A major goal in neuroscience is to understand how this is accomplished, and much excitement surrounds advancements in this area. The 1981 Nobel Prize in Physiology or Medicine was awarded in part to David Hubel and Torsten Wiesel for the discovery of neural cells that react to the size, shape, and orientation of visual stimuli \nolinebreak \cite{Hubel59}. The Nobel Prize in this area was awarded for study of neural activity once again in 2014, this time to John O'Keefe, May-Britt Moser, and Edvard Moser for their discovery of cells that act as a positioning system in the brain \nolinebreak \cite{Okeefe71}. They found that these place cells are used by the brain to create a map of the area around an organism. This is used as an ``expected'' environment that will be reactively changed when perception and expectation disagree.

In both of these experiments, convex codes---codes that may be realized by an arrangement of open convex sets in Euclidean space---were observed. These codes may be the key to how the brain represents relationships between stimuli, and they are the focus of our study. However, this definition of convex code relies on extrinsic data. The brain does not have information on the type of stimuli that provokes a neural response. Instead, it must rely on neural activity alone and interpret this activity as environmental stimuli \cite{Curto08}. Interpreting the neural activity of the hippocampal place cells separated from the stimuli that provoke it relies on determining the intrinsic properties that define convex codes. How can we determine if a neural code is convex based on neural activity alone? If a code is convex, what is the minimal dimension required to realize the code as a collection of convex open sets in Euclidean space?

An algebraic approach to this problem was introduced by Curto, Itskov, Veliz-Cuba, and Youngs in \cite{Curto13}. These methods were further expanded when Curto, \textit{et al.} in \cite{Curto15} introduced tools for intrinsically showing a neural code is convex and some conditions that prevent a convex realization. We will present a generalization of one of their results. In Section \ref{sec_bg}, we introduce convex codes and their minimal embedding dimensions. In Section \ref{sec_mainres}, we give a method of constructing a convex realization for codes whose codewords satisfy some incidence properties.


\section{Convex Neural Codes}\label{sec_bg}

A \textit{codeword} on $n$ neurons is a subset $\sigma \subset \lbrack n \rbrack \defeq \{1,2,\dots,n\}$ where the presence of $k$ in $\sigma$ signifies that the $k^\text{th}$ neuron is active. When there is no risk of confusion, we will denote a subset of $[n]$ as a string of its elements--- e.g. $\{1,3,4\}=134$. A \textit{neural code} is a collection of codewords, $\C = \{\sigma_j\}_{j \in \mathcal{J}}$. A collection of sets $\{U_i\}_{i\in \lbrack n \rbrack}$ in $\mathbb{R}^d$ defines a neural code 

\[\C\big(\{U_i\}_{i\in \lbrack n \rbrack} \big) \defeq \big\{ \sigma \subset [n] : 
\emptyset \neq \big( \bigcap_{i \in \sigma} U_i \big) \setminus \big( \bigcup_{j \notin \sigma} U_j \big) \big\}.\]

If $\C$ can be realized as $\C\big(\{U_i\}_{i\in \lbrack n \rbrack}\big)$ for a collection of convex open sets, then $\C$ is a \textit{convex code}. If $d$ is the smallest number for which such a collection exists, then the \textit{minimal embedding dimension} of $\C$, denoted $d(\C)$, is $d$. Not all codes are convex. For example, if $\C_0=\{12,23\}$, then a realization of $\C_0$ requires three convex open sets, $U_1, U_2,$ and $U_3$, such that $U_1 \cap U_2 \neq \emptyset$, $U_2 \cap U_3 \neq \emptyset$, $U_2 \subset U_1 \cup U_3$, and $U_1 \cap U_3 = \emptyset$. If $U_1, U_2, \text{ and } U_3$ are nonempty open sets that satisfy these relations, then $U_2$ is not even connected. On the other hand, $\C_1=\{1234, 123, 12, 2, 23, 234\}$ is a convex code as observed from the following realization:

\begin{figure}[ht]
\begin{center}
\begin{tikzpicture}
	\draw (0,0) ellipse (1 and .5);
	\node at (.3,.7) {$U_4$};
	\draw (0,0) circle (1.5);
	\node at (1.7,.5) {$U_3$};	
 	\draw (0,0) ellipse (3 and 2);
	\node at (3,1) {$U_2$};
  	\draw (-1, 0) ellipse (1.3 and .4);
	\node at (-2,.6) {$U_1$};
\end{tikzpicture}
 \caption{$\C_1 = \C\big( \{U_1,U_2,U_3,U_4\} \big)$}
\end{center}
\end{figure}
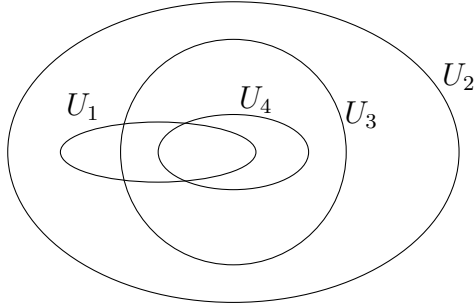

Given a neural code $\C$ and a codeword $\sigma \in \C$, we denote $\C\lvert_\sigma\defeq \{ \tau \in \C : \tau \subset \sigma \}$. The codeword $\mu$ is a \textit{maximal codeword} of $\C$ if there does not exist a codeword $\tau \in \C$ such that $\sigma \subsetneq \tau$. A code $\C$ is said to be \textit{max intersection-complete} if given maximal codewords $\mu_1,\mu_2,\dots,\mu_s \in \C$, then $\mu_1 \cap \mu_2 \cap \dots \cap \mu_s \in \C$. The presence of intersections of maximal codewords is connected to the convexity of the code, as seen below.

\begin{prop}[Proposition 4.6 from \cite{Curto15}] \label{prev result}
Let $\C$ be a code. If the intersection of any two distinct maximal codewords is empty, then $\C$ is convex and $d(\C) \leq 2$.
\end{prop}

\begin{theorem}[part of Theorem 4.4 from \cite{Cruz16}] \label{max result}
Let $\C$ be a max intersection-complete code with $s$ maximal codewords. Then $\C$ is a convex code and $d(\C) \leq \max\{2,s-1\}$.
\end{theorem}

\begin{remark}
Theorem \ref{max result} holds when considering codes that can be realized by closed convex sets as well \cite{Cruz16}.
\end{remark}

The proof given for Proposition \ref{prev result} involves inscribing a polygon in a circle and labeling the resulting partition with the codewords of $\C$. We will use a similar construction to build a convex realization of a larger class of neural codes that satisfy the following property: for every codeword $\tau \in \C$ and any collection of maximal codewords $\mu_1,\dots,\mu_t \in \C$, we have $\tau \cap \mu_1 \cap \dots \cap \mu_t \in \C$. Such a code is called \textit{strongly max intersection-complete}. It is clear that all strongly max intersection-complete codes are max intersection-complete as well, and thus convex by Theorem \ref{max result}. However, we present a different construction for this class of codes that results in a smaller upper bound for $d(\C)$. For our purposes, we may always assume that the empty codeword is present in $\C$. The absence of the empty codeword affects neither the convexity of $\C$ nor its minimal embedding dimension.


\section{Constructing a Convex Realization}\label{sec_mainres}

Our approach is to build several circles partitioned into regions, label the regions with codewords of $\C$, and connect the circles into a network that, in most places, looks like a hypercylinder. The structure of this network is determined by a graph. Let $G_\C$ denote the graph whose vertices are maximal codewords of $\C$ and whose edges are nonempty maximal codewords, then we have the following.

\begin{theorem} \label{main}
Let $\C$ be a strongly max intersection-complete neural code. If $G_\C$ is a quasi-square grid in $\mathbb{R}^d$, then $\C$ is convex and $d(\C)\leq d+2$.
\end{theorem}

\subsection{What is a quasi-square grid?}
We now recall some terminology from graph theory so that we can define quasi-square grid. Let $G$ be a simple graph with vertex set $V(G)$ and edge set $E(G)$. For any $V \subset V(G)$, let $G\lvert_V$ denote the graph that has vertex set $V$ and edges $e \in E(G\lvert_V)$ whenever $e \in E(G)$ and $e$ is incident to two vertices of $V$. Given a vertex $v \in V(G)$, the degree of $v$, denoted $\deg(v)$, is the number of edges incident to $v$.

A \textit{path} of length $m+1$ in $G$ is a set of distinct vertices $v_0,v_1,\dots,v_m$ and edges $e_1,e_2,\dots,e_m$ such that $e_i$ is incident to both $v_{i-1}$ and $v_i$. When an edge is incident to two vertices, we call those vertices \textit{adjacent}. A path where $v_0=v_m$ is called a \textit{$m$-cycle}. 

Our interest will be in graphs that can be drawn in a special way in Euclidean space. We identify vertices and edges in a graph with points and intervals in Euclidean space, respectively. A graph is called a \textit{square grid} in $\mathbb{R}^d$ if its vertex set is ${\mathbb{Z}^d \cap [a_1,b_1] \times \dots \times [a_d, b_d]}$, and two vertices are connected by an edge whenever they are at distance one. This may be generalized as follows: 
  
\begin{definition}
A graph is called a \textit{quasi-square grid} in $\mathbb{R}^d$ if its vertices correspond to a subset of points in $\mathbb{Z}^d$ and two vertices $x=(x_1,\dots,x_d)$ and $y=(y_1,\dots,y_d)$ are connected by an edge only if their coordinates differ in exactly one place, say $x_i < y_i$, and there is no vertex $(z_1,\dots,z_d)$ such that $x_i < z_i < y_i$ and $z_j=x_j=y_j$ for $j \neq i$.
\end{definition}

Equivalently, every quasi-square grid can be obtained by starting with a square grid and then performing any of the following operations in any order:
\begin{itemize}
	\item Any edge may be removed
	\item Any vertex along with all edges incident to it may be removed
	\item If there exists a path of three distinct vertices such that any pair differs in only one coordinate and the vertex in the middle is of degree two, then the vertex in the middle may be removed and the other two vertices may be joined by an edge
\end{itemize}

\begin{example}
In Figure \ref{qsgrid}, we obtain the quasi-square grid on the right from the square grid on the left by taking the following steps: 
\begin{enumerate}
	\item Remove vertex A
	\item Remove vertex E
	\item Remove the edge joining B with C
	\item Remove the edge joining B with D
	\item Remove vertex D while joining the vertices to either side of it with an edge
\end{enumerate}

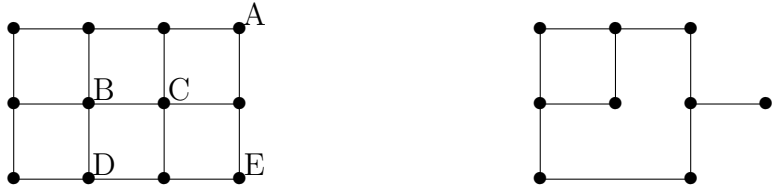
\begin{figure}
\begin{center}
\begin{tikzpicture}
 \draw[step=1cm] (0,0) grid (3,2);
 \foreach \Point in {(0,0), (1,0), (2,0), (3,0), (0,1), (1,1), (2,1), (3,1), (0,2), (1,2), (2,2), (3,2)}{
    \node at \Point {\textbullet};
    }
 \node at (1.2,0.2) {D};
 \node at (3.2,0.2) {E};
 \node at (1.2,1.2) {B};
 \node at (2.2,1.2) {C};
 \node at (3.2,2.2) {A};
 \foreach \Point in {(7,0), (9,0), (7,1), (8,1), (9,1), (10,1), (7,2), (8,2), (9,2)}{
    \node at \Point {\textbullet};
    }
 \draw[-] (7,0) -- (9,0);
 \draw[-] (7,0) -- (7,2);
 \draw[-] (7,1) -- (8,1);
 \draw[-] (7,2) -- (9,2);
 \draw[-] (8,2) -- (8,1);
 \draw[-] (9,0) -- (9,2);
 \draw[-] (9,1) -- (10,1);
\end{tikzpicture}
\caption{A square grid (left) and a quasi-square grid (right)} \label{qsgrid}
\end{center}
\end{figure}
\end{example}

\subsection{Attaching circles to $G_\C$}
Recall $G_\C$ is the graph whose vertices are the maximal codewords of $\C$ and two vertices are joined together by an edge precisely when the intersection of the maximal codewords is nonempty. When $G_\C$ is a quasi-square grid, it is a deformation retract of our desired construction. To create a convex realization of our code, we will attach a labeled partitioned circle to every point of the graph. Before describing how to partition and label the circles, we must define some constants. Suppose $\C$ is a code containing $s$ maximal codewords, $s^\prime$ pairs of which have nonempty intersection. Let $\mu_1, \dots, \mu_s$ be the maximal codewords and let $\sigma_{i,j}\defeq \mu_i \cap \mu_j$. We define the following:

\begin{itemize}
 \item $\lvert \sigma \rvert = \max_{\sigma_{i,j}\neq\emptyset} \#\big( \C\lvert_{\sigma_{i,j}} \big)-1$
 \item $ k_1= \max_{\mu_i} \#\big( \mu_i\setminus \bigcup_{j\neq i}\mu_j \big)$
 \item $k_2= \max_{\sigma_{i,j}\neq\emptyset} \#\big( \mu_i \setminus \sigma_{i,j} \big)$
 \item $r= 2^{k_1}-1+s^\prime \big((2^{k_2}-1)\lvert\sigma\rvert+(2^{k_2}-2)\big)$
\end{itemize}

In the following algorithm, when describing the partition of our circle by an inscribed polygon, we will refer to an interior region and several segments. By interior region we mean the part of the left circle labeled A in Figure \ref{circex}. By segments, we mean the parts of the left circle labeled B, C, and D. When we refer to the corresponding segment of a different circle, we mean the segment in the same relative position with respect to the other circle. In the below figure, C and E are corresponding segments.

\begin{figure}[H]
\begin{center}
 \circpicex{A}{B}{C}{D} \hspace{1in} \circpicex{}{}{E}{}
\end{center}
\caption{} \label{circex}
\end{figure}

\begin{alg}\label{alg} ~  \\
\textbf{Input:} A strongly max intersection-complete code $\C$ containing $s$ maximal codewords $\mu_i$, exactly $s^\prime$ pairs of which have nonempty intersection $\sigma_{i,j}$, such that $G_\C$ has no $3$-cycles \\
\textbf{Output:} A set of $s+s^\prime$ circles, denoted $S_1, S_2, \dots, S_s$ and $S_{i,j}$ for each nonempty $\sigma_{i,j}$, each partitioned into $r$ segments and an interior region bounded by those segments with the following properties:
\begin{enumerate}[label=(\alph*)]
	\item every codeword of $\C$ is a label of some region or segment \label{every_word_found}
	\item every region and segment either is labeled with a codeword of $\C$ or unlabeled \label{only_words_found}
	\item if $\tau$ is the label of a segment of a circle and $\eta$ is the label of the interior region of the same circle, then $\tau \subset \eta$ \label{int_is_big}
	\item if $\tau_i$ is the label of a segment of $S_i$ and $\tau_{i,j}$ is the label of the corresponding segment of $S_{i,j}$, then $\tau_i \cap \tau_{i,j} = \tau_i \cap \sigma_{i,j} = \mu_i \cap \tau_{i,j}$ where $\mu_i$ and $\sigma_{i,j}$ are the labels of the interior regions of circles $S_i$ and $S_{i,j}$, respectively \label{lines_up}
\end{enumerate}
\textbf{Algorithm:} 
\begin{enumerate}
	\item Construct $s+s^\prime$ disjoint circles in $\mathbb{R}^2$, one denoted $S_i$ for each maximal codeword $\mu_i$ and one denoted $S_{i,j}$ for each nonempty $\sigma_{i,j}$. \label{make circles}
	\item Inscribe an $r$-gon inside each of the circles. This partitions each circle into $r+1$ regions---$r$ segments enclosing an interior region.
	\item \underline{Add the $\mu_i$:} 
For each $S_i$, label the interior region $\mu_i$. \label{max_codewords}
	\item \underline{Add the $\sigma_{i,j}$:} 
For each $S_{i,j}$, label the interior region $\sigma_{i,j}$. \label{int_codewords}
	\item Label groups of segments of each circle in a manner such that corresponding segments of any two circles have the same group label: \label{split}
	\begin{enumerate}[label=(\roman*)]
		\item Make one unlabeled group of $2^{k_1}-1$ segments, \label{part1}
		\item Make $s^\prime$ groups of $2^{k_2}-2$ segments: one labeled $(i,j)^\prime$ for each $\sigma_{i,j}$, \label{part2}
		\item Make $s^\prime$ groups of $(2^{k_2}-1)\lvert\sigma\rvert$ segments: one labeled $(i,j)$ for each $\sigma_{i,j}$. \label{part3}
	\end{enumerate}
	\item \underline{Add codewords that are disjoint from every $\sigma_{i,j}$:} 
For every $S_i$, label the segments in the unlabeled group with the codewords $\tau$ that satisfy the conditions $\tau \in \C\lvert_{\mu_i}$ and $\tau \cap \sigma_{i,j} = \emptyset$ for all $j$ using each $\tau$ once. \label{disjoint_codewords}
	\item \underline{Add the codewords that are subsets of some $\sigma_{i,j}$:} 
For every $S_{i,j}$, choose $2^{k_2}-1$ blank segments from the group labeled $(i,j)$. Choose a codeword $\tau \subsetneq \sigma_{i,j}$ and label all of these segments $\tau$. Repeat this step for every $\tau \in \C\lvert_{\sigma_{i,j}}\setminus\{\sigma_{i,j}\}$. \label{subwords of sigma}
	\item \underline{Add the codewords that are supersets of some $\sigma_{i,j}$:}  For every $S_i$ and for all $j$ such that $\sigma_{i,j} \neq \emptyset$, label the segments in partition $(i,j)^\prime$ with the codewords $\tau \in \C\lvert_{\mu_i}$ such that $\sigma_{i,j} \subsetneq \tau \subsetneq \mu_i$ using each $\tau$ once. \label{superwords of sigma}
	\item \underline{Add missing codewords that are supersets of the codewords from step \ref{subwords of sigma}:} 
For each $S_i$ and for each $\tau \in \C\lvert_{\mu_i}$ such that $\tau \cap \sigma_{i,j} \neq \emptyset$, $\tau \nsubseteq \sigma_{i,j}$, and $\sigma_{i,j} \nsubseteq \tau$ for some $\sigma_{i,j}$, choose a blank segment of $S_i$ in the group labeled $(i,j)$ such that the corresponding segment in $S_{i,j}$ is labeled $\tau \cap \sigma_{i,j}$. Label the chosen segment of $S_i$ with $\tau$. \label{mediary_codewords}
	\item \underline{Ensure the circles ``properly align'':}
Label blank segments of $S_{i,j}$ with $\tau \cap \sigma_{i,j}$ if the corresponding segment in $S_i$ or $S_j$ is labeled $\tau$ and $\tau \cap \sigma_{i,j} \neq \emptyset$. Similarly, label blank segments of $S_i$ with $\tau$ if the corresponding segment of some $S_{i,j}$ is labeled $\tau \neq \emptyset$. Repeat until no additional segments are labeled. \label{fill in the blanks}
\end{enumerate}
\end{alg}

Before continuing, we must look carefully at step \ref{fill in the blanks} of the above algorithm. Since we label a segment based on the labels of corresponding segments of two other circles, we must be cautious that our choice of corresponding segment does not change the resulting label.

\begin{lemma}
Step \ref{fill in the blanks} of Algorithm \ref{alg} is a well-defined process.
\end{lemma}

\begin{proof}
We will break this problem into two cases: filling in a blank segment of $S_{k,\ell}$ and filling in a blank segment of $S_\ell$. In either case, note that the codewords added in step \ref{disjoint_codewords} never satisfy $\tau \cap \sigma_{i,j} \neq \emptyset$, therefore it is sufficient to consider only the segments that are in a group labeled $(i,j)$ or $(i,j)^\prime$ for some $i,j$.

Suppose a blank segment of $S_\ell$ in either group $(i,j)$ or group $(i,j)^\prime$ corresponds to the segment labeled $\tau$ in $S_{k,\ell}$ and the segment labeled $\tau^\prime$ in $S_{\ell,m}$ where $\tau$ and $\tau^\prime$ are both nonempty. We show this leads to a contradiction. By construction, if $\tau$ is a codeword in a segment of group $(i,j)$ and $(i,j)^\prime$, then $\tau \cap \mu_i \cap \mu_j \neq \emptyset$ and either $\tau \subset \mu_i$ or $\tau \subset \mu_j$. Since $\tau$ is a label of a segment of $S_{k,\ell}$, $\tau \subset \sigma_{k,\ell} = \mu_k \cap \mu_\ell$. Thus, either $\tau \subset \mu_i \cap \mu_k \cap \mu_\ell$ or $\tau \subset \mu_j \cap \mu_k \cap \mu_\ell$. By a similar argument, the same containments hold for $\tau^\prime$ if $\mu_k$ is replaced by $\mu_m$. If the intersection of three distinct maximal codewords is nonempty, then these three codewords form a $3$-cycle in $G_\C$. Therefore, either $(i,j)=(k,m)$ or $\ell \in \{i, j\}$. If $(i,j)=(k,m)$, then $\sigma_{k,m} \neq \emptyset$ and $\mu_k,\mu_\ell,\mu_m$ forms a $3$-cycle in $G_\C$. On the other hand, if $\ell \in \{i,j\}$, say $\ell=i$, then we have $\tau \cap \mu_\ell \cap \mu_j \neq \emptyset$ and $\tau \subset \mu_k \cap \mu_\ell$. Thus, $\mu_\ell \cap \mu_j \cap \mu_k \neq \emptyset$. Since $G_\C$ has no $3$-cycles, this forces $k=j$. However, applying the same argument with $\tau^\prime$ leads to $m=j$. This contradicts $k \neq m$.

Instead, suppose a blank segment of $S_{k,\ell}$ in either group $(i,j)$ or $(i,j)^\prime$ corresponds to both the segment labeled $\tau$ in $S_k$ and the segment labeled $\tau^\prime$ in $S_\ell$ where both $\tau \cap \sigma_{k,\ell}$ and $\tau^\prime \cap \sigma_{k,\ell}$ are nonempty. We show $\tau \cap \sigma_{k,\ell} = \tau^\prime \cap \sigma_{k,\ell}$ in this case. Since $\tau$ is a label of a segment in group $(i,j)$ or group $(i,j)^\prime$, $\tau \subset \mu_i$ or $\tau \subset \mu_j$. Moreover, since $\tau \cap \sigma_{k,\ell} \neq \emptyset$ and $\sigma_{k,\ell}= \mu_k \cap \mu_\ell$, we must have at least one of $\mu_i \cap \mu_k \cap \mu_\ell$ and $\mu_j \cap \mu_k \cap \mu_\ell$ nonempty. Thus, either $i \in \{k,\ell\}$ or $j \in \{k,\ell\}$. 

Without loss of generality, assume $i=k$. Similarly, either $\tau^\prime \subset \mu_k$ or $\tau^\prime \subset \mu_j$. If $\tau^\prime \subset \mu_k$, then $\tau^\prime \subset \sigma_{k,\ell}$. Since $\tau^\prime \cap \mu_m \subset \mu_k \cap \mu_\ell \cap \mu_m = \emptyset$ for all $m \notin \{k,\ell\}$, $\tau^\prime$ is not the label for any segment in a group labeled $(k,j)$ or $(k,j)^\prime$ unless $j=\ell$. On the other hand, if $\tau^\prime \subset \mu_j$, then we have $\mu_j \cap \mu_k \cap \mu_\ell \neq \emptyset$. In this case, we also have $j=\ell$.

We have reduced the problem to the case where $\tau$ and $\tau^\prime$ are labels of corresponding segments in either the $(k,\ell)$ group or the $(k,\ell)^\prime$ group. If they are labels of segments in $(k,\ell)$ group, the corresponding segment of $S_{k,\ell}$ is not unlabeled. In fact, the segments of $S_k$ and $S_\ell$ could have only been labeled when comparing them with $S_{k,\ell}$ during step \ref{fill in the blanks}. Therefore, $\tau$ and $\tau^\prime$ are labels in the $(k,\ell)^\prime$ group. However, this means that $\tau$ and $\tau^\prime$ were added in step \ref{superwords of sigma}. Thus, $\tau \cap \sigma_{k,\ell} = \tau^\prime \cap \sigma_{k,\ell} = \sigma_{k,\ell}$.
\end{proof}

Now that we are confident that there is no ambiguity in Algorithm \ref{alg}, we turn our attention to verifying we get an output and that it matches what we expect.

\begin{proof}[Proof of Correctness of Algorithm \ref{alg}]
First we show that our algorithm terminates. The only possible obstruction to termination is the last step of the algorithm which is repeated until no additional segments of any circle are labeled. However, since we have finitely many circles partitioned into finitely many regions, this condition must be satisfied in finite time.

Next, we want to show \ref{every_word_found}. Note that every element of $\C$ is in exactly one of the following four sets: \\
$A=\{\mu_i, \sigma_{i,j} : 1\leq i,j \leq s\}$ \\
$B=\{\tau : \tau \subsetneq \sigma_{i,j} \text{ for some } i,j\}$ \\
$C=\{\tau : \sigma_{i,j} \subsetneq \tau \subsetneq \mu_i \text{ for some } i,j\}$ \\
$D=\{\tau : \tau \nsubseteq \sigma_{i,j} \text{ and } \sigma_{i,j} \nsubseteq \tau \text{ for any } i,j \}$

The codewords in $A$ are added to our construction in steps \ref{max_codewords} and \ref{int_codewords}, the codewords in $B$ are added in step \ref{subwords of sigma}, and the codewords in $C$ are added in step \ref{superwords of sigma}. For every $\tau \in D$, either there exists some $i,j$ such that $\tau \cap \sigma_{i,j} \neq \emptyset$ or there does not. If such an $i,j$ exists, then $\tau$ is added in step \ref{mediary_codewords}. Otherwise, $\tau$ is added in step \ref{disjoint_codewords}. Thus, all codewords of $\C$ appear as a label at least once.

To show \ref{only_words_found}, we note that the only time we label a region by a codeword that was not specifically chosen from the list of elements of $\C$ is in step \ref{fill in the blanks}. Here, the labels appearing that are not already in our construction elsewhere are of the form $\tau \cap \sigma_{i,j} = \tau \cap \mu_i \cap \mu_j$ for some $\tau$ already appearing in our construction, and thus already in $\C$. However, since $\C$ is strongly max intersection-complete, this codeword is necessarily an element of $\C$. Thus, every labeled region of our construction is labeled with a codeword of $\C$.

Next, we prove \ref{int_is_big}. The segments of the circles are only labeled in steps \ref{disjoint_codewords}-\ref{fill in the blanks}. In steps \ref{disjoint_codewords}-\ref{mediary_codewords}, condition \ref{int_is_big} is satisfied by construction. In step \ref{fill in the blanks}, when we label a segment of $S_{i,j}$, the label is a subword of $\sigma_{i,j}$ by construction. When we label a segment of $S_i$ with $\tau$, then $\tau$ is already a label of $S_{i,j}$. Since we have already proven \ref{int_is_big} for the $S_{i,j}$, then we have $\tau \subset \sigma_{i,j} \subset \mu_i$ as desired.

Finally, we show the output satisfies \ref{lines_up}. We consider two cases: when at least one of the $\tau$ is the empty codeword and when neither $\tau$ is empty. In the former case, note that \ref{lines_up} is satisfied if $\tau_i=\tau_{i,j}=\emptyset$. Step \ref{fill in the blanks} ensures that we are never in the situation where $\tau_i=\emptyset$ and $\tau_{i,j}\neq \emptyset$. Moreover, if $\tau_i \neq \emptyset$ and $\tau_{i,j}=\emptyset$, then step \ref{fill in the blanks} requires us to label $\tau_{i,j}$ with $\tau_i \cap \sigma_{i,j}$. Thus, we must have $\emptyset = \tau_{i,j} = \tau_i \cap \sigma_{i,j}$ as desired. 

In the latter case, we note that the only steps where such a nonempty $\tau$ is added to the construction when the other is already nonempty are steps \ref{mediary_codewords} and \ref{fill in the blanks}. In step \ref{mediary_codewords}, we have $\tau_{i,j}=\tau_i \cap \sigma_{i,j}$ by construction, therefore $\tau_i \cap \tau_{i,j} = \tau_i \cap \sigma_{i,j}$. Moreover, since $\tau_{i,j} \subset \tau_i \subset \mu_i$, we have $\tau_i \cap \tau_{i,j} = \mu_i \cap \tau_{i,j}$ as well. In step \ref{fill in the blanks}, we either have $\tau_{i,j}=\tau_i \cap \sigma_{i,j}$ or $\tau_i=\tau_{i,j}$. In either case, it is immediate that $\tau_i \cap \tau_{i,j} = \tau_i \cap \sigma_{i,j} = \mu_i \cap \tau_{i,j}$.
\end{proof}

\subsection{A convex realization}
Now that we have labeled circles, we will use $G_\C$ to complete the construction. Suppose that $G_\C$ is a quasi-square grid in $\mathbb{R}^d$ and let $\pi_d: \mathbb{R}^{d+2} \rightarrow \mathbb{R}^d$ be the projection that forgets the first two coordinates. Then our construction is a $(d+2)$-dimensional construct with image $G_\C$ under $\pi_d$ and the fiber over a given point, $x$, is the circle $S_i$ if $\lVert x - \mu_i \rVert_\infty \leq \frac{1}{3}$ or the circle $S_{i,j}$ if it is on the edge connecting the vertices $\mu_i$ and $\mu_j$ but does not satisfy the above inequality for either vertex. Finally, for every $i \in [n]$, we define $U_i$ to be the interior of the union of every region of our construction whose label contains $i$.

\begin{example} We use Algorithm \ref{alg} to find a convex realization for the following code (where maximal codewords and the intersection of maximal codewords are in bold)
\[
\C_2=\{\mathbf{1234}, 12, 2, 124, 234, 134, \mathbf{34}, 4, \mathbf{345}, 5, 45\}
\]
\begin{center} 
	$G_{\C_2}=$
	\begin{tikzpicture}
	\node[draw](left) at (0,.2) {1234};
	\node[draw](right) at (2.3,.2) {345};
	\draw (left) -- (right) node[midway, above]{34};
	\end{tikzpicture}
\end{center}
Looking over our constants, we have $s^\prime=1, \lvert \sigma \rvert=1, \text{ and } k_1=k_2=2$, thus $r=8$. We will now follow give a step-by-step construction of a convex realization of $\C$ in $\mathbb{R}^3$ using our method. For the sake of brevity, we will use the subscript labels as written in the algorithm (i.e.\ $\sigma_{1,2}=34$). Based on the above graph,  we construct a cylinder whose circular cross-sections are partitioned and labeled as follows:

\vspace{\baselineskip}
\textbf{Steps \ref{make circles}-\ref{int_codewords}} \\
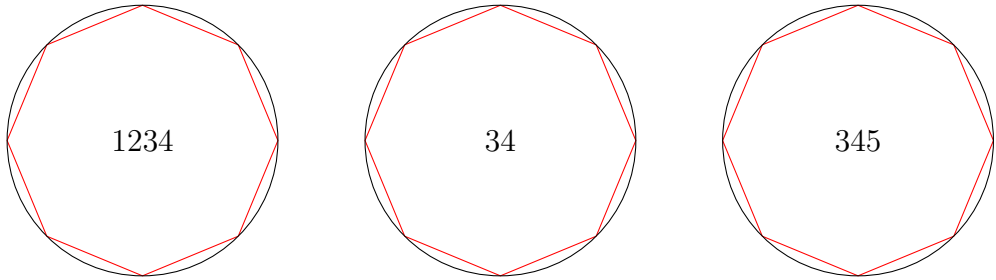
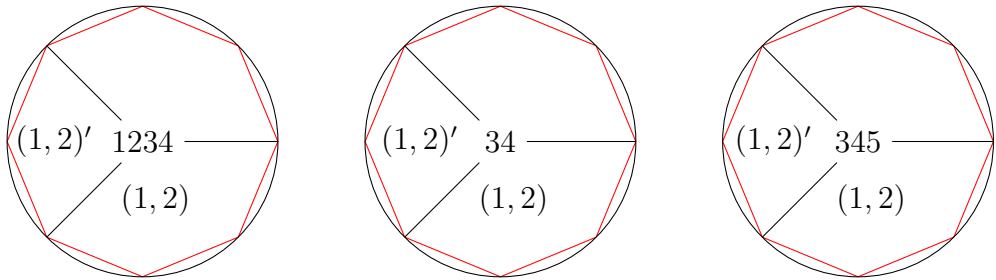
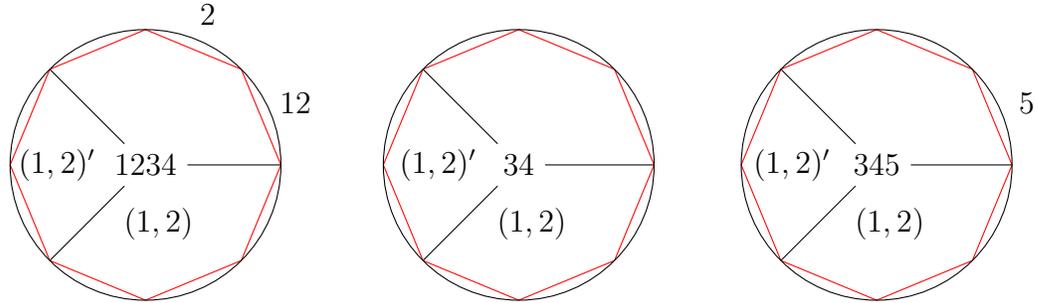
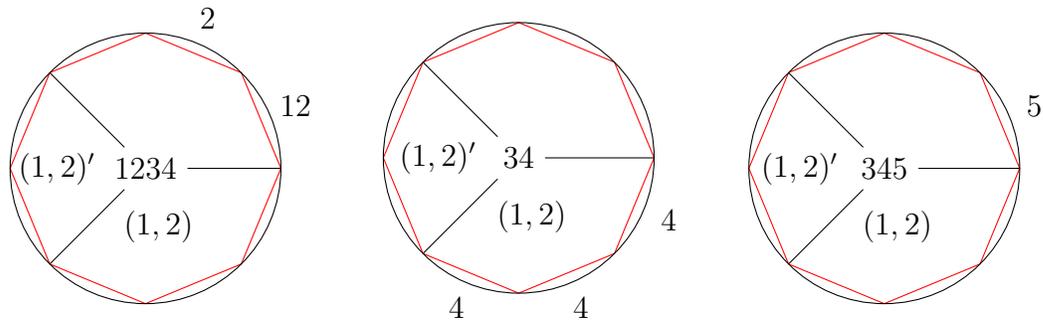
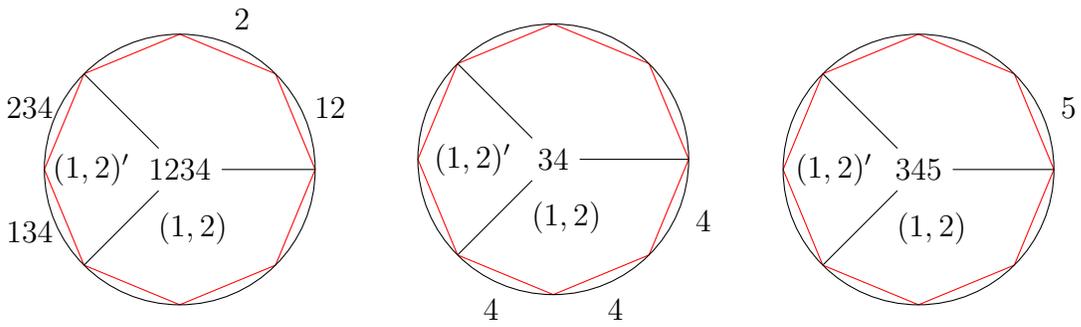
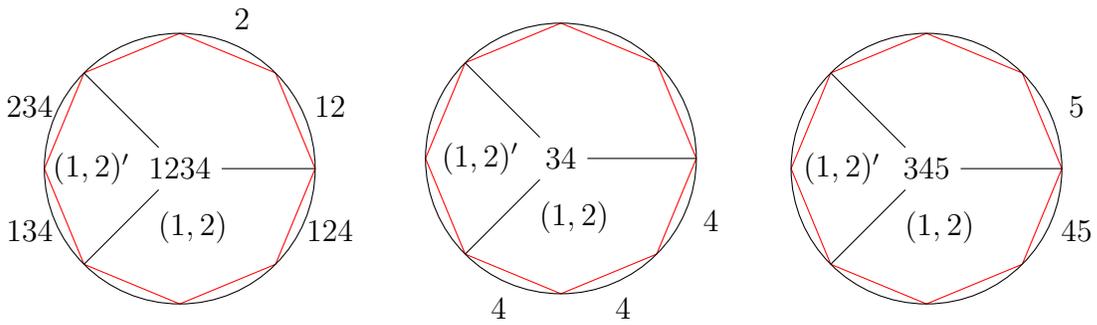
\begin{tabular}{ccc}
	\circpicnodiv{$1234$}{}{}{}{}{}{}{}{} &
	\circpicnodiv{$34$}{}{}{}{}{}{}{}{} &
	\circpicnodiv{$345$}{}{}{}{}{}{}{}{}
\end{tabular}

\textbf{Step \ref{split}} \\
\begin{tabular}{ccc}
	\circpic{$1234$}{}{}{}{}{}{}{}{} &
	\circpic{$34$}{}{}{}{}{}{}{}{} &
	\circpic{$345$}{}{}{}{}{}{}{}{}
\end{tabular}

\textbf{Step \ref{disjoint_codewords}} \\
\begin{tabular}{ccc}
	\circpic{$1234$}{$12$}{$2$}{}{}{}{}{}{} &
	\circpic{$34$}{}{}{}{}{}{}{}{} &
	\circpic{$345$}{$5$}{}{}{}{}{}{}{}
\end{tabular}

\textbf{Step \ref{subwords of sigma}} \\
\begin{tabular}{ccc}
	\circpic{$1234$}{$12$}{$2$}{}{}{}{}{}{} &
	\circpic{$34$}{}{}{}{}{}{$4$}{$4$}{$4$} &
	\circpic{$345$}{$5$}{}{}{}{}{}{}{}
\end{tabular}

\textbf{Step \ref{superwords of sigma}} \\
\begin{tabular}{ccc}
	\circpic{$1234$}{$12$}{$2$}{}{$234$}{$134$}{}{}{} &
	\circpic{$34$}{}{}{}{}{}{$4$}{$4$}{$4$} &
	\circpic{$345$}{$5$}{}{}{}{}{}{}{}
\end{tabular}

\textbf{Step \ref{mediary_codewords}} \\
\begin{tabular}{ccc}
	\circpic{$1234$}{$12$}{$2$}{}{$234$}{$134$}{}{}{$124$} &
	\circpic{$34$}{}{}{}{}{}{$4$}{$4$}{$4$} &
	\circpic{$345$}{$5$}{}{}{}{}{}{}{$45$}
\end{tabular}

\textbf{Step \ref{fill in the blanks}- first pass} \\
\begin{tabular}{ccc}
	\circpic{$1234$}{$12$}{$2$}{}{$234$}{$134$}{$4$}{$4$}{$124$} &
	\circpic{$34$}{}{}{}{$34$}{$34$}{$4$}{$4$}{$4$} &
	\circpic{$345$}{$5$}{}{}{}{}{$4$}{$4$}{$45$}
\end{tabular}

\textbf{Step \ref{fill in the blanks}- second pass} \\
\begin{tabular}{ccc}
	\circpic{$1234$}{$12$}{$2$}{}{$234$}{$134$}{$4$}{$4$}{$124$} &
	\circpic{$34$}{}{}{}{$34$}{$34$}{$4$}{$4$}{$4$} &
	\circpic{$345$}{$5$}{}{}{$34$}{$34$}{$4$}{$4$}{$45$}
\end{tabular}

\textbf{Completed Construction} \\
\begin{tabular}{ccc}
	\circpicnodiv{$1234$}{$12$}{$2$}{}{$234$}{$134$}{$4$}{$4$}{$124$} &
	\circpicnodiv{$34$}{}{}{}{$34$}{$34$}{$4$}{$4$}{$4$} &
	\circpicnodiv{$345$}{$5$}{}{}{$34$}{$34$}{$4$}{$4$}{$45$}
\end{tabular}
\end{example}

As the above example shows, there may be redundant segments that remain empty for the entirety of the algorithm. These segments are considered to be points that correspond to the empty codeword or may be left out of the construction entirely. Now that we have some familiarity with the content of Theorem \ref{main}, we turn our attention its proof.

\begin{proof}[Proof of Theorem \ref{main}]
Assume that $\C$ is a strongly max intersection-complete neural code such that $G_\C$ is a connected quasi-square grid in $\mathbb{R}^d$. Note that the definition of quasi-square grid prevents the presence of $3$-cycles in $G_\C$. Let $\mu_1, \dots, \mu_s$ be the vertices of  $G_\C$ and $\sigma_{i,j}$ be $\mu_i \cap \mu_j$. We will now construct a convex realization of $\C$ in $\mathbb{R}^{d+2}$ by labeling regions of a geometric object with numbers and then defining $U_i$ to be the interior of the union of all regions that contain the number $i$ in its label.

Let $G^\prime=\{y : \lVert y - x\rVert_\infty \leq \frac{1}{3} \text{ for some } x\in G_\C\}$ where $x=(x_1,\dots,x_d)$ and $y=(y_1,\dots,y_d)$. Let $P_x = \hat{D} \times x_1 \times \dots \times x_d$ where $x \in G^\prime$ and $\hat{D}$ is a disk centered at $(0,0)$ of radius $\frac{1}{3}$, and let $P= \cup_{x\in G^\prime} P_x$. In each $P_x \subset P$, inscribe a regular $r$-gon. Follow Algorithm \ref{alg} to obtain several labeled circles. If $\lVert x-y \rVert_\infty \leq \frac{1}{3}$ where $y$ is the coordinate of the vertex of $\mu_i$, then label $P_x$ as $S_i$. Otherwise, label $P_x$ as $S_{i,j}$ when $\lVert x-y \rVert_\infty \leq \frac{1}{3}$ where $y$ is a coordinate of $\sigma_{i,j}$.

We now confirm that the interior of the union of all regions with an $i$ in their label forms a convex set. First, we note that no $i$ is contained in $\sigma_{j,k} \cap \sigma_{\ell,m}$ when $\{j,k\} \neq \{\ell,m\}$. If we could find such $\sigma_{j,k},\sigma_{\ell,m}$ such that at least three of $j,k,\ell,m$ are distinct, say $j, k, \text{ and } \ell$, then we would have $\{i\} \subset \mu_j \cap \mu_k \cap \mu_\ell$. However, this would imply that $G_\C$ contains the three-cycle $\mu_j, \mu_k, \mu_\ell$ which contradicts that $G_\C$ is a quasi-square grid. Thus, the regions with an $i$ in their label must either be contained in either a single $S_j$ or contained in exactly one triple $S_j, S_{j,k}, \text{ and } S_k$. If the regions are all contained in a single $S_i$, then this is the cross product of the construction as used in the proof of proposition \ref{prev result} with connected intervals. Similarly, if $i$ is contained in exactly $S_j, S_{jk}, \text{ and } S_k$, then Algorithm \ref{alg} ensures that the regions containing $i$ in their label for each of these pieces are in the same location. Thus, our set has the same shape as it does in the case where $i$ only appears in a single $S_j$. Hence, each $U_i$ is convex. To complete our proof, we note that $d(\C) \leq \dim (P)= d+2$.
\end{proof}

\begin{remark}
Theorem \ref{main} holds with the same bound on dimension if we consider codes realized with only closed convex sets. The same construction and argument apply with the additional step of taking the closure of each $U_i$. 
\end{remark}

Note that Proposition \ref{prev result} is the special case of Theorem \ref{main} where $G_\C$ is a graph with no edges. Furthermore, $d(\C)$ is bounded by one less than the number of maximal codewords in $\C$ in Theorem \ref{max result} while the bound given in Theorem \ref{main} is usually much smaller.

\begin{cor}
	If $\C$ is strongly max intersection-complete and $G_\C$ is a union of disjoint paths, then $\C$ is convex and $d(\C) \leq 3$.
\end{cor}

\begin{proof}
	We shall draw $G_\C$ as a quasi-square grid in $\mathbb{R}$. Choose a vertex $v_0$ such that $\deg(v_0)=1$ and place it at $0$. We then proceed by induction: if we have placed the vertex $v_m$ at $m$ and $v_m$ is adjacent to another vertex, say $v_{m+1}$, then we place $v_{m+1}$ at $m+1$. Repeat this process for each connected component of $G_\C$ translating the graph so that the connected components of the graph do not overlap one another. We then apply Theorem \ref{main} with $d=1$.
\end{proof} 

Checking that $G_\C$ is a quasi-square grid may not be easy in all cases. Aside from the case that $G_\C$ contains a $3$-cycle, it is not immediately clear what properties will prevent the rigid graph structure that Theorem \ref{main} requires. However, the following corollary shows us that if $G_\C$ fails to be a quasi-square grid, the obstruction will be found in the cycle structure of the graph.

\begin{cor}
	Let $\C$ be a strongly max intersection-complete neural code. If $G_\C$ has no $3$-cycles and any two distinct cycles of $G_\C$ are disjoint (i.e., no two cycles have a vertex in common), then $\C$ is a convex code and 
\[ d(C) \leq D = 2+\lvert V(G_\C) \rvert + \#\{\text{cycles in }G_\C \} - \#\{\text{vertices contained in some cycle in }G_\C \}. \]
\end{cor}

\begin{proof}
Suppose that $G_\C$ is a connected graph with no $3$-cycles and any two cycles of $G_\C$ are disjoint. We will show that $G_\C$ is a quasi-square grid then appeal to Theorem \ref{main}. Let $v_0 \in V(G_\C)$. Place $v_0$ at the origin in $\mathbb{R}^{D-2}$ and call this graph $\hat{G}$. Note that $\hat{G}$ is a quasi-square grid. We will now proceed inductively. 
	
Suppose $V(\hat{G}) = \{v_0,v_1,\dots,v_{m-1}\}$, $\hat{G}=G_\C\lvert_{V(\hat{G})}$, and $\hat{G}$ is a quasi-square grid. Additionally, suppose that if $\hat{G}$ contains at least two vertices that are in some cycle of $G_\C$, then $\hat{G}$ contains all of the vertices in that cycle. Denote the coordinates of the vertex $v_{m-1}$ in $\hat{G}$ by $(x_1,\dots,x_k,0,\dots,0)$. If $V(G_\C)=V(\hat{G})$, then we are done. Otherwise, since $G_\C$ is connected, there is some vertex $v_m\in V(G_\C)\setminus V(\hat{G})$ such that $v_m$ is adjacent to $v_i$ for some $i<m$ in $G_\C$. Without loss of generality, assume that $v_m$ is adjacent to $v_{m-1}$. First we note that $v_m$ must not be adjacent to any vertex $v_j$ with $j<m-1$, otherwise $v_m, v_{m-1}, v_j$ would be three vertices in some cycle of $G_\C$ and hence $v_m$ would already be a vertex in $\hat{G}$. Either $v_{m-1}$ and $v_{m}$ are two vertices in a cycle of $G_\C$ or they are not. If $v_{m-1}$ and $v_m$ are \textit{not} two vertices in the same cycle, then place $v_m$ at $(x_1,x_2,\dots,x_k,1,0,\dots,0)$ and draw the edge connecting $v_{m-1}$ and $v_m$ as the line segment with end points at their respective positions. We now have $\hat{G}$ is a quasi-square grid with $V(\hat{G}) = \{v_0,v_1,\dots,v_m\}$, $\hat{G}=G_\C\lvert_{V(\hat{G})}$, and that if $\hat{G}$ contains at least two vertices that are in some cycle of $G_\C$, then $\hat{G}$ contains all of the vertices in that cycle.
	
On the other hand, if $v_{m-1}$ and $v_{m}$ are in the same cycle in $G_\C$, let $v_{m-1}, v_m, \dots, v_{m+\ell}$
 be the vertices in the cycle labeled such that $v_i$ and $v_j$ are adjacent to one another if 
 $i=j-1$ or both $i=m-1$
 and $j=m+\ell$. Since $G_\C$ contains no $3$-cycles, we must have $\ell \geq 2$. Place the vertices in the following manner: for $i\in \{0,1,\dots,\ell-2\}$ we place $v_{m+i}$ at $(x_0, x_1, \dots, x_k+1+i, 0, \dots 0)$, $v_{m+\ell-1}$ at
  $(x_0, x_1, \dots, x_k+\ell-1,1,0,\dots,0)$, and $v_{m+\ell}$ at $(x_0,x_1,\dots,x_k,1,0,\dots,0)$. We then add in the edges between the vertices that are adjacent in $G_\C$ by drawing the line segment with endpoints at the respective vertices' locations. We now have $\hat{G}$ is a
   quasi-square grid with $V(\hat{G}) = \{v_0,v_1,\dots,v_{m+\ell}\}$ $\hat{G}=G_\C\lvert_{V(\hat{G})}$, and that if $\hat{G}$ contains at least two vertices that are in some cycle of $G_\C$, then $\hat{G}$ contains all of the vertices in that cycle. 
	
Since $G_\C$ is a finite graph, we must get $G_\C=\hat{G}$ after finitely many steps. Furthermore, when constructing $\hat{G}$, we used an additional coordinate whenever we added an entire cycle or a vertex that is not contained in any cycle. Thus, we required no more than $D-2$ coordinates, and our graph is contained in $\mathbb{R}^{D-2}$ as desired. Therefore, we may apply Theorem \ref{main} to attain a convex realization in $\mathbb{R}^D$. When $G_\C$ is not a connected graph, we repeat the above process for each connected component of $G_\C$, translating the graph resulting from each connected component so that the components do not overlap.
\end{proof}

\subsection*{Acknowledgements}
This work began as a project for Seminar in Mathematical Biology which was jointly taught by Anne Shiu and Jay Walton. RW was partially supported by an NSF grant (DMS-1501370) awarded to Frank Sottile. The author thanks Anne Shiu and Frank Sottile for helpful discussions.

\bibliographystyle{amsplain}
\bibliography{mybib}

\providecommand{\bysame}{\leavevmode\hbox to3em{\hrulefill}\thinspace}
\providecommand{\MR}{\relax\ifhmode\unskip\space\fi MR }
\providecommand{\MRhref}[2]{%
  \href{http://www.ams.org/mathscinet-getitem?mr=#1}{#2}
}
\providecommand{\href}[2]{#2}
\begin{thebibliography}{1}

\bibitem{Cruz16}
J.~{Cruz}, C.~{Giusti}, V.~{Itskov}, and B.~{Kronholm}, \emph{{On open and
  closed convex codes}},  (2016), arXiv:1609.03502.

\bibitem{Curto15}
C.~Curto, E.~Gross, J.~Jeffries, K.~Morrison, M.~Omar, Z.~Rosen, A.~Shiu, and
  N.~Youngs, \emph{What makes a neural code convex?},  (2015),
  arXiv:1508.00150.

\bibitem{Curto08}
C.~Curto and V.~Itskov, \emph{Cell groups reveal structure of stimulus space},
  PLoS Comput. Biol. \textbf{4} (2008), no.~10, e1000205, 13. \MR{2457124}

\bibitem{Curto13}
C.~Curto, V.~Itskov, A.~Veliz-Cuba, and N.~Youngs, \emph{The neural ring: an
  algebraic tool for analyzing the intrinsic structure of neural codes}, Bull.
  Math. Biol. \textbf{75(9)} (2013), 1571--1611.

\bibitem{Hubel59}
D.~Hubel and T.~Wiesel, \emph{Receptive fields of single neurons in the cat's
  striate cortex}, J. Physiol. \textbf{148(3)} (1959), 574--591.

\bibitem{Okeefe71}
J.~O'Keefe and J.~Dostrovsky, \emph{The hippocampus as a spatial map.
  {P}reliminary evidence from unit activity in the freely-moving rat}, Brain
  Res. \textbf{34(1)} (1971), 171--175.

\end{thebibliography}
\end{document}